\documentclass[12pt]{article}
\usepackage{graphicx}
\usepackage{amsmath}
\usepackage{amsfonts}
\usepackage{amsthm}
\usepackage[T1]{fontenc}
\usepackage{url}
\usepackage{color}
\usepackage[margin=1in]{geometry}
\usepackage{amssymb,bm}
\usepackage{amsmath}
\usepackage{amsthm}
\usepackage{graphicx}
\usepackage[active]{srcltx} 
\usepackage{hyperref}
\hypersetup{pdfborder=0 0 0}

\newtheorem{theorem}{{\sc Theorem}}[section]

\newtheorem{lemma}[theorem]{{\sc Lemma}}

\def\XXint#1#2#3{{\setbox0=\hbox{$#1{#2#3}{\int}$ }
\vcenter{\hbox{$#2#3$ }}\kern-.6\wd0}}

\newcommand{\Gk}{\kappa}

\newcommand{\Gth}{\theta}

\bmdefine\BGa{\alpha}
\bmdefine\BGb{\beta}
\bmdefine\BGd{\delta}
\bmdefine\BGe{\epsilon}
\bmdefine\BGve{\varepsilon}
\bmdefine\BGf{\phi}
\bmdefine\BGvf{\varphi}
\bmdefine\BGg{\gamma}
\bmdefine\BGc{\chi}
\bmdefine\BGi{\iota}
\bmdefine\BGk{\kappa}
\bmdefine\BGl{\lambda}
\bmdefine\BGn{\eta}
\bmdefine\BGm{\mu}
\bmdefine\BGv{\nu}
\bmdefine\BGp{\pi}
\bmdefine\BGth{\theta}
\bmdefine\BGvth{\vartheta}
\bmdefine\BGr{\rho}
\bmdefine\BGvr{\varrho}
\bmdefine\BGs{\sigma}
\bmdefine\BGvs{\varsigma}
\bmdefine\BGt{\tau}
\bmdefine\BGj{\tau}
\bmdefine\BGu{\upsilon}
\bmdefine\BGo{\omega}
\bmdefine\BGx{\xi}
\bmdefine\BGy{\psi}
\bmdefine\BGz{\zeta}
\bmdefine\BGD{\Delta}
\bmdefine\BGF{\Phi}
\bmdefine\BGG{\Gamma}
\bmdefine\BGL{\Lambda}
\bmdefine\BGP{\Pi}
\bmdefine\BGT{\Theta}
\bmdefine\BGS{\Sigma}
\bmdefine\BGU{\Upsilon}
\bmdefine\BGO{\Omega}
\bmdefine\BGX{\Xi}
\bmdefine\BGY{\Psi}

\bmdefine\BCA{{\mathcal A}}
\bmdefine\BCB{{\mathcal B}}
\bmdefine\BCC{{\mathcal C}}
\bmdefine\BCD{{\mathcal D}}
\bmdefine\BCE{{\mathcal E}}
\bmdefine\BCF{{\mathcal F}}
\bmdefine\BCG{{\mathcal G}}
\bmdefine\BCH{{\mathcal H}}
\bmdefine\BCI{{\mathcal I}}
\bmdefine\BCJ{{\mathcal J}}
\bmdefine\BCK{{\mathcal K}}
\bmdefine\BCL{{\mathcal L}}
\bmdefine\BCM{{\mathcal M}}
\bmdefine\BCN{{\mathcal N}}
\bmdefine\BCO{{\mathcal O}}
\bmdefine\BCP{{\mathcal P}}
\bmdefine\BCQ{{\mathcal Q}}
\bmdefine\BCR{{\mathcal R}}
\bmdefine\BCS{{\mathcal S}}
\bmdefine\BCT{{\mathcal T}}
\bmdefine\BCU{{\mathcal U}}
\bmdefine\BCV{{\mathcal V}}
\bmdefine\BCW{{\mathcal W}}
\bmdefine\BCX{{\mathcal X}}
\bmdefine\BCY{{\mathcal Y}}
\bmdefine\BCZ{{\mathcal Z}}

\bmdefine\Bzr{ 0}
\bmdefine\Ba{ a}
\bmdefine\Bb{ b}
\bmdefine\Bc{ c}
\bmdefine\Bd{ d}
\bmdefine\Be{ e}
\bmdefine\Bf{ f}
\bmdefine\Bg{ g}
\bmdefine\Bh{ h}
\bmdefine\Bi{ i}
\bmdefine\Bj{ j}
\bmdefine\Bk{ k}
\bmdefine\Bl{ l}
\bmdefine\Bm{ m}
\bmdefine\Bn{ n}
\bmdefine\Bo{ o}
\bmdefine\Bp{ p}
\bmdefine\Bq{ q}
\bmdefine\Br{ r}
\bmdefine\Bs{ s}
\bmdefine\Bt{ t}
\bmdefine\Bu{ u}
\bmdefine\Bv{ v}
\bmdefine\Bw{ w}
\bmdefine\Bx{ x}
\bmdefine\By{ y}
\bmdefine\Bz{ z}
\bmdefine\BA{ A}
\bmdefine\BB{ B}
\bmdefine\BC{ C}
\bmdefine\BD{ D}
\bmdefine\BE{ E}
\bmdefine\BF{ F}
\bmdefine\BG{ G}
\bmdefine\BH{ H}
\bmdefine\BI{ I}
\bmdefine\BJ{ J}
\bmdefine\BK{ K}
\bmdefine\BL{ L}
\bmdefine\BM{ M}
\bmdefine\BN{ N}
\bmdefine\BO{ O}
\bmdefine\BP{ P}
\bmdefine\BQ{ Q}
\bmdefine\BR{ R}
\bmdefine\BS{ S}
\bmdefine\BT{ T}
\bmdefine\BU{ U}
\bmdefine\BV{ V}
\bmdefine\BW{ W}
\bmdefine\BX{ X}
\bmdefine\BY{ Y}
\bmdefine\BZ{ Z}

\date{}
\begin{document}
\title{The sharp $L^p$ Korn interpolation and second inequalities in thin domains}
\author{D. Harutyunyan}
\maketitle

\begin{abstract}
In the present paper we extend the $L^2$ Korn interpolation and second inequalities in thin domains, proven in [\ref{bib:Harutyunyan.3}], to the space $L^p$ for any $1<p<\infty.$ A thin domain in space is roughly speaking a shell with non-constant thickness around a smooth enough two dimensional surface. The inequality that we prove in $L^p$ holds for practically any thin domain $\Omega\subset\mathbb R^3$ and any vector field $\Bu\in W^{1,p}(\Omega).$ The constants in the estimate are asymptotically optimal in terms of the domain thickness $h.$ This in particular solves the problem of finding the asymptotics of the optimal constant in the classical Korn second inequality in $L^p$ for thin domains in terms of the domain thickness in almost full generality. The remarkable fact is that the interpolation inequality reduces the problem of estimating the gradient $\nabla\Bu$ in terms of the strain $e(\Bu)$ to the easier problem of estimating only the vector field $\Bu$, which is a Korn-Poincar\'e inequality.
\end{abstract}

\section{Introduction}
\label{sec:1}
Assume $h>0$ is a small parameter and assume $S\subset \mathbb R^3$ is a connected and compact $C^3$ surface with a unit normal $\Bn(x)$\footnote{The surface $S$ does not have to be orientable.} for $x\in S.$ While a shell of thickness $h$ is the $h/2$ neighborhood of $S$ in the normal direction, i.e., it is given by $\Omega=\{x+t\Bn(x) \ : \ x\in S,\ t\in (-h/2,h/2)\},$ where the surface $S$ is called the mid-surface of the shell    
$\Omega,$ a family of thin domains $\Omega^h$ with thickness of order $h$ is defined in terms of Lipschitz functions $g_1^h(x),g_2^h(x)\colon S\to (0,\infty)$,
as follows: 
\begin{equation}
\label{1.1}
\Omega^h=\{x+t\Bn(x) \ : \ x\in S,\ t\in (-g_1^h(x),g_2^h(x))\},
\end{equation}
where the functions $g_1^h$ and $g_2^h$ are assumed to satisfy the uniform conditions
\begin{equation}
\label{1.2}
h\leq g_1^h(x),g_2^h(x)\leq c_1 h,\quad \text{and}\quad |\nabla g_1^h(x)|+|\nabla g_2^h(x)|\leq c_2h,\quad\text{for all}\quad x\in S,
\end{equation}
to ensure that the thickness of $\Omega^h$ is of order $h$ and does not have rapid oscillations as $h\to 0.$ The problem of determining rigidity of thin domains is more than a century old in nonlinear elasticity. The problem has been solved for plates only\footnote{Or for shells that have a flat part} by Friesecke, James and M\"uller  [\ref{bib:Fri.Jam.Mue.1},\ref{bib:Fri.Jam.Mue.2}]. The term "rigidity" here is the geometric rigidity of a thin domain, which is defined in terms of the geometric rigidity estimate of Friesecke, James and M\"uller, that reads as follows: \textit{Assume $\Omega\subset\mathbb R^3$ is open bounded connected and Lipschitz. Then there exists a constant $C_I=C_I(\Omega),$ such that for every vector field $\Bu\in H^1(\Omega),$ there exists a constant proper rotation $\BR\in SO(3)$, such that}
\begin{equation}
\label{1.3}
\|\nabla\Bu-\BR\|_{L^2(\Omega)}^2\leq C_{I}\int_\Omega\mathrm{dist}^2(\nabla\Bu(x),SO(3))dx.
\end{equation}
If $\Omega$ is a thin domain, then the optimal value of the constant $C_I$ in (\ref{1.3}) typically has the asymptotic form $C_I=c_Ih^{-\alpha},$ where 
$\alpha\geq 0$ and the constant $c_I>0$ depends only on the mid-surface $S$ and the Lipschitz characters of the surrounding faces $g_1^h$ and $g_2^h.$
The value\footnote{The question of whether such a value exists is generally open.} of $\alpha$ then identifies the rigidity of $\Omega.$ Depending on the problem under consideration the vector field $\Bu\in H^1(\Omega)$ may or may not satisfy boundary or normalization conditions. In the case when $\Bu$ is not required to satisfy any additional conditions other than the integrability $\Bu\in H^1(\Omega),$ one typically has for the exponent $\alpha>0,$ i.e., the constant $C_I$ indeed blows up in the vanishing thickness limit [\ref{bib:Fri.Jam.Mue.1},\ref{bib:Fri.Jam.Mue.2}], in particular one has $\alpha=2$ for plates. One can always rotate the field $\Bu,$ thus assume without loss of generality that $\BR=\BI$ in (\ref{1.3}). The linearization of (\ref{1.3}) around the identity is Korn's first inequality without boundary conditions [\ref{bib:Korn.1},\ref{bib:Korn.2},\ref{bib:Kon.Ole.1},\ref{bib:Kohn.1}] which reads as follows: \textit{Assume $\Omega\subset\mathbb R^3$ is open bounded connected and Lipschitz. Then there exists a constant $C_{II}=C(\Omega),$ depending only on $\Omega,$ such that for every vector field $\Bu\in H^1(\Omega)$ there exists a skew-symmetric matrix $\BA\in \mathbb R^{3\times 3,}$ i.e., $\BA+\BA^T=0,$ such that
\begin{equation}
\label{1.4}
\|\nabla\Bu-\BA\|_{L^2(\Omega)}^2\leq C_{II}\|e(\Bu)\|_{L^2(\Omega)}^2,
\end{equation}
where $e(\Bu)=\frac{1}{2}(\nabla\Bu+\nabla\Bu^T)$ is the symmetrized gradient (the strain in linear elasticity).} It is a well-known fact that one can pass from (\ref{1.3}) to (\ref{1.4}) and vice versa\footnote{However the current techniques for the passage from (\ref{1.4}) to (\ref{1.3}) rely on local truncation of $\Bu$ and do not preserve the asymptotics of the constant $C_{II}.$} If the field $\Bu$ is prescribed on the thin face of the $\Omega,$ then the asymptotics of 
$C_{II}$ is known 
[\ref{bib:Gra.Har.1},\ref{bib:Gra.Har.2},\ref{bib:Gra.Har.3},\ref{bib:Harutyunyan.2}]. 
Traditionally (\ref{1.4}) is proven by means of Korn's second inequality [\ref{bib:Korn.1},\ref{bib:Korn.2},\ref{bib:Kon.Ole.2}], which imposes no boundary or normalization condition on the vector field $\Bu\in H^1(\Omega)$ and reads a follows: \textit{Assume $\Omega\subset\mathbb R^n$ is open bounded connected and Lipschitz. Then there exists a constant $C_2=C_2(\Omega),$ depending only on $\Omega,$ such that for every vector field $\Bu\in H^1(\Omega)$ there holds:}
\begin{equation}
\label{1.5}
\|\nabla\Bu\|_{L^2(\Omega)}^2\leq C_2(\|\Bu\|_{L^2(\Omega)}^2+\|e(\Bu)\|_{L^2(\Omega)}^2).
\end{equation}
A new inequality, called Korn's first-and-a-half inequality (later renamed Korn's interpolation inequality as it interpolates between Korn's first and second inequalities) was introduced in [\ref{bib:Gra.Har.1}] and employed in [\ref{bib:Gra.Har.3},\ref{bib:Harutyunyan.2}] to study the inequality (\ref{1.4}) for shells. 
An asymptotically sharp version of the interpolation estimate [\ref{bib:Harutyunyan.3}, Theroem~3.1] was recently proven by the author for practically any thin domains $\Omega$ and any vector fields $\Bu\in H^1(\Omega),$ see also [\ref{bib:Harutyunyan.4}]. The new interpolation inequality solves two problems: 1. It is stronger then Korn's second inequality and solves the problem of finding the asymptotics of the constant $C_2$ in (\ref{1.5}), yielding $C_2=ch^{-1}$ for thin domains $\Omega.$ 2. It reduces the problem of proving (\ref{1.4}) to proving a Korn-Poincar\'e estimate on the vector field $\Bu$ with $e(\Bu)$ in place of $\nabla\Bu.$ In the present work we extend the interpolation estimate to the space $L^p$ for any $1<p<\infty.$ We do not adopt the classical strategy of proving $L^p$ estimates out of the $L^2$ ones for elliptic operators by proving a weak type $L^1$ estimate and doing Marcinkewich interpolation, but we do it directly by redoing the proof in [\ref{bib:Harutyunyan.3},\ref{bib:Harutyunyan.4}] and modifying where necessary. Let us mention that while some of the modifications are trivial, some are highly not and require new ideas and tricks. The main difficulties occur in two key Lemmas 4.2 and 5.1. The first difficulty for instance occurs when one tries to prove the analogue of the main Lemma~4.2 in [\ref{bib:Harutyunyan.3}] in $L^p$ as the self-duality of $L^2$ has been heavily used in the original proof. In particular, in many steps it was essential that the functions under consideration belong to $L^2$ as throughout the proof of Lemma~4.2 in [\ref{bib:Harutyunyan.3}], a PDE was multiplied by a test function from the same function space and the result was integrated by parts, which certainlty does not work for $p\neq 2,$ as well as a modification of test functions (to make them belong to the same space $L^p$) does not seem to work. We will point the main issues out when proving an analogues estimate, or will present a sketch of proof, referring to [\ref{bib:Harutyunyan.3}] for details if the same proof works.

\section{Notation}
\setcounter{equation}{0}
\label{sec:2}

In this section we recall the main notation and definitions, from [\ref{bib:Harutyunyan.3}]. Going back to the setting (\ref{1.1})-(\ref{1.2}), we assume that the mid-surface $S\subset\mathbb R^3$ is connected, compact, regular and of class $C^3$ up to its boundary. Another technical assumption that we make is that locally, and up to boundary, $S$ has a parametrization by means of the principal variables $\Gth$ and $z.$ From the compactness of $S$ we can then extract a finite atlas of patches $S\subset\cup_{i=1}^m\Sigma_i$ such that each patch $\Sigma_i$ can be parametrized by the principal variables $z$ and $\Gth$ ($z=$constant and $\Gth=$constant are the principal lines on $\Sigma_i$) that change in the ranges $z\in [z_i^1(\Gth),z_i^2(\Gth)]$ for $\Gth\in [0,\omega_i],$ where $\omega_i>0$ for $i=1,2,\dots,m.$ 
We also assume that the functions $z_i^1(\Gth)$ and $z_i^2(\Gth)$ satisfy the conditions
\begin{equation}
\label{2.1}
\min_{1\leq i\leq m}\inf_{\Gth\in [0,\omega_i]}[z_i^2(\Gth)-z_i^1(\Gth)]=l>0,
\end{equation}
which roughly speaking means that each patch does not have an infinitesimally sharp edge in the principal directions. Since there will be no condition imposed on the vector field $\Bu\in H^1(\Omega),$ (see Theorem~\ref{th:3.1}), we can prove the interpolation inequality locally (over a single patch) and then sum the obtained estimates in $i=1,2,\dots,m.$ Assume therefore in what follows that $S$ is a single patch parametrized as $\Br=\Br(\Gth,z)$ in the principal variables with the new notation $\omega:=\omega_1.$ Denote the metric $A_{z}=\left|\frac{\partial \Br}{\partial z}\right|, A_{\Gth}=\left|\frac{\partial \Br}{\partial\Gth}\right|$ on $S$ and let $\Gk_{z}$ and $\Gk_{\Gth}$ be the two principal curvatures. In what follows we will mainly use the notation $f_{,\alpha}$ for the partial derivative $\frac{\partial f}{\partial\alpha}$ inside gradient matrices. Also, for the normal to $S$ direction $\Bn$ we will use the variable $t.$ Thus we have for any $\Bu=(u_t,u_\Gth,u_z)\in H^1(\Omega,\mathbb R^3)$ the formula (\ref{bib:Tov.Smi.})
\begin{equation}
\label{2.2}
\nabla\Bu=
\begin{bmatrix}
  u_{t,t} & \dfrac{u_{t,\Gth}-A_{\Gth}\Gk_{\Gth}u_{\Gth}}{A_{\Gth}(1+t\Gk_{\Gth})} &
\dfrac{u_{t,z}-A_{z}\Gk_{z}u_{z}}{A_{z}(1+t\Gk_{z})}\\[3ex]
u_{\Gth,t}  &
\dfrac{A_{z}u_{\Gth,\Gth}+A_{z}A_{\Gth}\Gk_{\Gth}u_{t}+A_{\Gth,z}u_{z}}{A_{z}A_{\Gth}(1+t\Gk_{\Gth})} &
\dfrac{A_{\Gth}u_{\Gth,z}-A_{z,\Gth}u_{z}}{A_{z}A_{\Gth}(1+t\Gk_{z})}\\[3ex]
u_{ z,t}  & \dfrac{A_{z}u_{z,\Gth}-A_{\Gth,z}u_{\Gth}}{A_{z}A_{\Gth}(1+t\Gk_{\Gth})} &
\dfrac{A_{\Gth}u_{z,z}+A_{z}A_{\Gth}\Gk_{z}u_{t}+A_{z,\Gth}u_{\Gth}}{A_{z}A_{\Gth}(1+t\Gk_{z})}
\end{bmatrix},
\end{equation}
in the orthonormal local basis $(\Bn,\Be_\Gth,\Be_z).$ It is convenient to prove the estimates for the gradient restricted to the mid-surface denoted by $\BF$ (which is obtained from (\ref{2.2}) by omitting the small terms $t\Gk_\Gth$ and $t\Gk_z$ in the denominators of the second and third columns of $\nabla\Bu$),
and then pass from $\BF$ to the actual gradient $\nabla\Bu$ utilizing the obvious bounds
$\|e(\BF)-e(\Bu)\|_{L^p(\Omega)}\leq \|\BF-\nabla\Bu\|_{L^p(\Omega)}\leq 2\min(\|\BF\|_{L^p(\Omega)},\|\nabla\Bu\|_{L^p(\Omega)}).$

\section{Main results}
\label{sec:3}
\setcounter{equation}{0}
Before formulating the main theorem, let us introduce the domain mid-surface and thickness parameters, which are the quantities 
$\omega,l,L,Z,a,A,c_1,c_2$ and $k,$ where $\omega$ is defined in the previous section, $l$ is defined in (\ref{2.1}) for the situation of one patch ($m=1$), and  
\begin{align}
\label{3.1}
&a:=\min_{D}(A_\Gth,A_z)>0, \quad A:=\|A_\Gth\|_{W^{2,\infty}(D)}+\|A_z\|_{W^{2,\infty}(D)}<\infty,\\ \nonumber
&k:=\|\Gk_\Gth\|_{W^{1,\infty}(D)}+\|\Gk_z\|_{W^{1,\infty}(D)}<\infty,\\ \nonumber
&L:=\max_{\Gth\in [0,\omega]}[z^2(\Gth)-z^1(\Gth)]<\infty,
\quad Z:=\left(\|z^1\|_{W^{1,\infty}[0,\omega]}+\|z^2\|_{W^{1,\infty}[0,\omega]}\right)<\infty,
\end{align}
where we are assuming there is only one parametrization patch with $z^1=z_1^1,$ $z^2=z_1^2$ and $D=\{(\Gth,z)\ : \ \Gth\in [0,\omega], z\in[z^1(\Gth),z^2(\Gth)]\}.$ In what follows the constants $h_0>0$ and $C>0$ will depend only on the exponent $p$ and the domain mid-surface and thickness parameters. We will use the notation $\|f\|_p$ for the $L^p$ norm skipping the domain of consideration whenever it creates no ambiguity. The following theorem is the Korn interpolation inequality in $L^p.$ 
\begin{theorem}[Korn's interpolation inequality in $L^p$]
\label{th:3.1}
Let $1<p<\infty$ and assume conditions (\ref{2.1}) and (\ref{3.1}) hold. Then there exists constants $h_0,C>0,$ such that Korn's interpolation inequality holds:
\begin{equation}
 \label{3.2}
\|\nabla\Bu\|_p^2\leq C\left(\frac{\|u_t\|_p\cdot\|e(\Bu)\|_p}{h}+\|\Bu\|_p^2+\|e(\Bu)\|_p^2\right),
\end{equation}
for all $h\in(0,h_0)$ and $\Bu=(u_t,u_\Gth,u_z)\in W^{1,p}(\Omega),$ where $\Bn$ is the unit normal to the mid-surface $S.$ Moreover, the exponent of $h$ in the inequality (\ref{3.2}) is optimal for any thin domain $\Omega$ satisfying (\ref{2.1}) and (\ref{3.1}), i.e., there exists a displacement $\Bu\in W^{1,p}(\Omega,\mathbb R^3)$ realizing the asymptotics of $h$ in (\ref{3.2}).
\end{theorem}

The next theorem, that is a consequence of Theorem~\ref{th:3.1}, provides a sharp Korn's second inequality for thin domains.

\begin{theorem}[Korn's second inequality in $L^p$]
\label{th:3.2}
Let $1<p<\infty$ and assume conditions (\ref{2.1}) and (\ref{3.1}) hold. There exists constants $h_0,C>0,$ such that Korn's second inequality holds:
\begin{equation}
  \label{3.3}
\|\nabla\Bu\|_p^2\leq \frac{C}{h}(\|\Bu\|_p^2+\|e(\Bu)\|_p^2),
\end{equation}
for all $h\in(0,h_0)$ and $\Bu=(u_t,u_\Gth,u_z)\in W^{1,p}(\Omega).$ Moreover, the exponent of $h$ in the inequality (\ref{3.3}) is optimal for any thin domain $\Omega$ satisfying (\ref{2.1}) and (\ref{3.1}), i.e., there exists a displacement $\Bu\in W^{1,p}(\Omega,\mathbb R^3)$ realizing the asymptotics of $h$ in (\ref{3.3}).
\end{theorem}
An important remark is as follows. It has been proven in [\ref{bib:Gra.Har.3},\ref{bib:Harutyunyan.2}], that in the presence of zero Dirichlet boundary conditions on the displacement field $\Bu$ on the thin part of the boundary $\partial\Omega,$ the scaling of the optimal (in terms of the asymptotics in the thickness $h$)
constant $C_{II}$ in (\ref{1.4}) depends only on the principal curvatures, (and in fact Gaussian curvature if it has a constant sign) of the domain mid-surface $S.$ Thus we consider below 
the main domain-curvature situations below.
\begin{itemize}
\item[1.] Parabolic surfaces: $\Gk_\Gth=0$ and $|\Gk_z|>0$ on $S.$ This includes cylindrical and conical surfaces with convex cross-sections.
\item[2.] Elliptic surfaces: $\Gk_\Gth\Gk_z>0$ on $S.$   
\item[3.] Hyperbolic surfaces: $\Gk_\Gth\Gk_z<0$ on $S.$
\end{itemize}

What follows below within the ongoing section is straightforward to check by simple algebraic calculations. The Ansatz given at the end of Section~5 confirms the sharpness of both (\ref{3.2}) and (\ref{3.3}). However it realizes the asymptotics of the optimal constant in the rigidity estimate (\ref{1.4}) only for elliptic surfaces. Meanwhile, Ans\"atze realizing the upper bound part in (\ref{1.4}) for 
parabolic and elliptic surfaces $S$ are constructed in [\ref{bib:Gra.Har.3}] and [\ref{bib:Harutyunyan.2}] respectively. It is remarkable that both of them realize the asymptotics of constants in the main
estimate (\ref{3.2}), and at the same time do not realize the asymptotics of the best constant in (\ref{3.3}). This being said, while Korn's second inequality is classical, the interpolation inequality (\ref{3.2}) seems to be the "best" asymptotic Korn second-like inequality holding true and being sharp for all main domain-curvature situations. This suggests that when the interpolation inequality is utilized for thin domains to reduce the problem of estimating $\nabla \Bu$ in terms of $e(\Bu),$ to estimating just the field $\Bu$ in terms of $e(\Bu),$ the constants are not expected to suffer a loss in the asymptotics in $h.$ Therefore, this would provide a significant reduction of complexity of (\ref{1.4}).

\section{Inequalities in two dimensions}
\label{sec:4}
\setcounter{equation}{0}

First of all let us emphasize that in the ongoing section the constant $C>0$ may depend only on $p,$ $C_1$ and $C_2$ (introduced below in Lemma~\ref{lem:4.1}) unless otherwise specified. As pointed out in the introduction, we will follow the analysis in [\ref{bib:Harutyunyan.3}]. The following is a rigidity estimate for harmonic functions in two dimensional thin domains.
\begin{lemma}
\label{lem:4.1}
Assume $b>0,$ $h\in (0,b/2)$ and let the Lipschitz functions $\varphi_1,\varphi_2\colon[0,b]\to (0,\infty)$ and the constants $C_1,C_2>0$ be such that 
\begin{equation}
\label{4.1}
h\leq \varphi_i(y)\leq C_1h,\quad |\nabla\varphi_i(y)|\leq C_2h \quad\text{for all}\quad y\in [0,b], \ \ i=1,2.
\end{equation}
Denote the thin domain $D=\{(x,y)\in\mathbb R^2 \ : \ y\in (0,b), x\in (-\varphi_1(y),\varphi_2(y))\}$ that has a thickness of order $h.$ Then there exists a constant $C>0,$ such that any harmonic function $w\in C^2(\bar D)$ fulfills the inequality
\begin{equation}
\label{4.2}
\inf_{a\in\mathbb R}\|\partial_y w-a\|_{L^p(D)}\leq\frac{Cb}{h}\|\partial_x w\|_{L^p(D)}.
\end{equation}
 \end{lemma}
 
 \begin{proof}
The proof is based on a use of Korn's first inequality in bulk and a clever choice of vector fields entering it, and follows the lines of the same lemma in [\ref{bib:Harutyunyan.3}], thus we will only present a sketch with little detail.\\
\textbf{Proposition.} \textit{There exists a constant $C>0,$ such that for any vector field $\BW=(u,v)\colon D\to\mathbb R^2,$ there exists a skew-symmetric matrix $\BA\in\mathbb R^2,$ such that}
\begin{equation}
\label{4.3}
\|\nabla\BW-\BA\|_{L^p(D)}\leq \frac{Cb}{h}\|e(\BW)\|_{L^p(D)}.
\end{equation}
By a variable change $(x,y)\to (bx',by')$ we can assume that $b=1.$ We adopt the localization argument of Kohn and Vogelius  [\ref{bib:Koh.Vog.}] that was also successfully employed in [\ref{bib:Fri.Jam.Mue.1},\ref{bib:Fri.Jam.Mue.2}]. Denote the positive whole number $N=\left[\frac{1}{h}\right]+1$ and consider the horizontal cuts $D_k$ of $D$ by the lines $y=\frac{k-1}{N}$ and $y=\frac{k+1}{N}$ for $k=1,2,\dots,N-1.$ As each of the rescaled domains $D_k$ is of order $h$ with a piecewise Lipschitz boundary, Korn's first inequality [\ref{bib:Kon.Ole.2}, Theorem~6], that is invariant under 
variable recaling $x\to \lambda x',$ gives 
\begin{equation}
\label{4.4}
\|\nabla \BW-\BA_k\|_{L^p(D_k)}\leq C\|e(\BW)\|_{L^p(D_k)},\quad k=1,2,\dots,N-1.
\end{equation}
for some skew-symmetric matrices $\BA_k\in \mathbb R^{2\times 2},$ $k=1,2,\dots, N-1.$ Assume $N>2$ otherwise done. For $1\leq k \leq N-2$ we have by the triangle inequality and by (\ref{4.4}) the estimate 
\begin{align*}
\|\BA_k-\BA_{k+1}\|_{L^p(D_k\cap D_{k+1})}&\leq \|\nabla \BW-\BA_k\|_{L^p(D_k)} +\|\nabla \BW-\BA_{k+1}\|_{L^p(D_{k+1})}\\
&\leq C\|e(\BW)\|_{L^p(D_k)}+C\|e(\BW)\|_{L^p(D_{k+1})},
\end{align*}
thus as the measure of $D_k\cap D_{k+1}$ is of order $h^2,$ we get 
\begin{equation}
\label{4.5}
|\BA_k-\BA_{k+1}|\leq \frac{C}{h^{2/p}}(\|e(\BW)\|_{L^p(D_k)}+\|e(\BW)\|_{L^p(D_{k+1})}), \quad k=1,2,\dots,N-2.
\end{equation}
The last inequality yields by the triangle inequality the estimate 
\begin{equation}
\label{4.6}
\|\BA_1-\BA_k\|_{L^p(D_k)} \leq C\|e(\BW)\|_{L^p(D)}, \quad k=1,2,\dots,N-1.
\end{equation}
Therefore choosing $\BA=\BA_1$ for (\ref{4.3}), by an application of the triangle inequality and owing back to (\ref{4.4}) and (\ref{4.6}), we arrive at the desired estimate (\ref{4.3}).\\
It remains to note that an application of the proposition for the specific vector field $\BW=(u,v),$ where
\begin{equation}
\label{4.7}
u(x,y)=w(x,y),\quad\text{ and}\quad v(x,y)=-\int_0^x \partial_y w(t,y)dt+\int_0^y \partial_x w(0,z)dz,
\end{equation}
with $a=a_{12}$ of the matrix $\BA$ in (\ref{4.3}), will derive (\ref{4.2}) from (\ref{4.3}). 

\end{proof}

The next lemma is the key estimate in the analysis.
\begin{lemma}
\label{lem:4.2}
Let $b>0$ and $h\in\left(0,b/8\right).$ Let $\varphi_1,\varphi_2,C_1,C_2$ and the domain $D$ be as in Lemma~\ref{lem:4.1}. Then there exists a constant $C>0,$ such that any harmonic function $w\in C^2(\bar D)$ fulfills the inequality
\begin{equation}
\label{4.8}
\|\partial_y w\|_{L^p(D)}^2\leq C\left(\frac{1}{h}\|w\|_{L^p(D)}\cdot\|\partial_x w\|_{L^p(D)}+\frac{1}{b^2}\|w\|_{L^p(D)}^2+\|\partial_x w\|_{L^p(D)}^2\right).
\end{equation}
\end{lemma}

\begin{proof}
In the proof of the foregoing lemma all $\|\cdot\|_p$ norms will be the norm $\|\cdot\|_{L^p(D)}$ unless specified otherwise. Again, noting that (\ref{4.8}) is invariant under variable scale 
$(x,y)\to \lambda (x',y'),$ we can assume $b=1.$ We divide the proof into three steps.\\

\textbf{Step 1. An estimate on the narrowed domain.} \textit{There exists a constant $C>0$ such that any harmonic function $w\in C^2(\bar D)$ fulfills the inequality:}
\begin{equation}
\label{4.9}
\|\partial_y w\|_{L^p\left((-h/2,h/2)\times(0,1)\right)}^2\leq C\left(\frac{1}{h}\|w\|_p\cdot\|\partial_x w\|_p+\|w\|_p^2+\|\partial_x w\|_p^2\right).
\end{equation}

\begin{proof}[Proof of Step 1.]

Let us first comment on the challenges of the case $p\neq 2$ in contrast to $p=2.$ Note that the proof technique in [15] of Step1 for $p=2$ was first multiplying the equality $\Delta w=0$ by a test function (basically $w$) that is in $L^2,$ and then integrating by parts. However, this of course does not work for $p\neq 2,$ thus we found the nontrivial replacement of that part by the interpolation inequality (\ref{4.10}). Next, there are several additional terms resulting from (\ref{4.10}) that one has to appropriately control (for instance $J_1$) which is another nontrivial task. The localization technique together with the classical Caciapolli inequality in $L^p$ come to rescue, see below.\\
For a fixed $z\in [h,1/4]$ let $\varphi(y)\colon[0,1]\to[0,1]$ be a smooth cutoff function supported on $(z,1-z)$ such that
$\varphi(y)=1$ for $y\in[2z,1-2z],$ $|\varphi'(y)|\leq \frac{2}{z}$ and $|\varphi''(y)|\leq \frac{2}{z^2}$ for $y\in[0,1].$ 
For any $x\in I_h=(-h/2,h/2)$ we have the interpolation inequality 
\begin{equation}
\label{4.10}
\left(\int_0^1 |\partial_y(\varphi w(x,y))|^p dy\right)^2\leq 
C\int_0^1 |\varphi w(x,y)|^p dy \int_0^1 |\partial^2_y(\varphi w(x,y))|^p dy,
\end{equation}
thus taking into account the choice of $\varphi$ and the inequality $|x_1+x_2+x_3|^p\leq 3^{p-1}(|x_1|^p+|x_2|^p+|x_3|^p),$ we get integrating (\ref{4.10}) in $x$ over the interval $I_h$ the estimate
\begin{equation}
\label{4.11}
\left(\int_{I_h\times(2z,1-2z)} |\partial_y w|^p \right)^2\leq C(J_1+J_2+J_3+J_4),
\end{equation}
where
\begin{align}
\label{4.12}
J_1&=\int_{I_h\times(z,1-z)} |w|^p \int_{I_h\times(z,1-z)} |\partial^2_y w|^p,\quad J_2=\left(\frac{1}{z^{p}} \int_{I_h\times(z,1-z)}|w|^p \right)^2,\\ \nonumber
J_3&=\frac{1}{z^{p}} \int_{I_h\times(z,1-z)} |w|^p \int_{I_h\times(z,2z)}|\partial_y w|^p,\quad 
J_4=\frac{1}{z^{p}} \int_{I_h\times(z,1-z)} |w|^p \int_{I_h\times(1-2z,1-z)} |\partial_yw|^p.
\end{align}
The trickiest part is estimating the summand $J_1$ in (\ref{4.11}). Note that the domain $D$ contains the rectangle $(-h,h)\times (0,1),$ thus since $z\geq h,$ 
we can cover the rectangle $I_h\times(z,1-z)$ by a sequence of balls $B_i$ (discs in this case) with radii $r=h/\sqrt{2}$ and centers $O_i=(0,y_i),$ where $y_i=z+2ir,$ for $i=1,2,\dots,k$ with 
$k=[(1-z)/2r]$ and an additional $(k+1)-$th ball $B_{k+1}=B_r(0,1-z)$ if necessary. For any point $(x,y)\in B_i,$ the distance of $(x,y)$ from $\partial D$ is a least $h(1-1/\sqrt{2}),$ thus by the Caccioppoli inequality for harmonic functions in $L^p$ we have in any ball $B_i$ the estimate
\begin{equation}
\label{4.13}
\int_{B_i}|\nabla (\partial_x w)|^p\leq \frac{C}{r^p}\int_{B_i} |\partial_x w|^p=\frac{C}{h^p}\int_{B_i} |\partial_x w|^p,
\end{equation}
summing which in $i$ we discover 
\begin{equation}
\label{4.14}
\int_{I_h\times(z,1-z)}|\nabla (\partial_x w)|^p\leq \frac{C}{h^p}\int_{D} |\partial_x w|^p.
\end{equation}
Next we have by the harmonicity of $w,$ that $\partial^2_{y}w=-\partial^2_xw,$ thus we have by (\ref{4.14}) the bounds
\begin{align*}
\int_{I_h\times(z,1-z)} |\partial^2_{y} w|^p&=\int_{I_h\times(z,1-z)} |\partial^2_x w|^p\\ 
&\leq \int_{D}|\nabla(\partial_x w)|^p\\
 & \leq \frac{C}{h^p}\int_D |\partial_{x}w|^p.
\end{align*}
From the last estimate we obtain
\begin{equation}
\label{4.15}
J_1\leq C\left(\frac{1}{h}\|w\|_p\|\partial_{x} w\|_p\right)^p.
\end{equation}
For $J_2$ we have the obvious inequality 
\begin{equation}
\label{4.16}
J_2\leq \frac{1}{z^{2p}} \|w\|_p^{2p}. 
\end{equation}
For $J_3$ and $J_4$ we have by the Cauchy inequality 
\begin{equation}
\label{4.17}
J_3+J_4\leq \frac{1}{2\epsilon^2z^{2p}}\|w\|_p^{2p}+\epsilon^2 \left(\int_{I_h\times(z,2z)}|\partial_y w|^p\right)^2
+\epsilon^2 \left(\int_{I_h\times(1-2z,1-z)}|\partial_y w|^p\right)^2
\end{equation}
where $\epsilon>0$ is a parameter yet to be chosen. Combining (\ref{4.11}) and (\ref{4.15})-(\ref{4.17}), we arrive at 
\begin{align}
\label{4.18}
\left(\int_{I_h\times(2z,1-2z)} |\partial_y w|^p \right)^2& \leq C\left(\frac{1}{h}\|w\|_p\|\partial_x w\|_p\right)^p+\frac{C}{z^{2p}}\left(1+\frac{1}{\epsilon^2}\right)\|w\|_p^{2p}\\ \nonumber
&+C\epsilon^2\left(\int_{I_h\times(z,2z)} |\partial_y w|^p \right)^2+C\epsilon^2\left(\int_{I_h\times(1-2z,1-z)} |\partial_y w|^p \right)^2
\end{align}
Once one has the bound (\ref{4.18}), which plays the role of (4.14) in [\ref{bib:Harutyunyan.3}], the rest of Step1 follows the appropriate lines in [\ref{bib:Harutyunyan.3}]. 
We present the details here for the convenience of the reader. Denote for the sake of brevity 
\begin{align*}
D_{z}^{bot}&=\{(x,y) \ : \ y\in(0,z), x\in (-\varphi_1(y),\varphi_2(y))\},\\
D_{z}^{top}&=\{(x,y) \ : \ y\in(1-z,1), x\in (-\varphi_1(y),\varphi_2(y))\}.
\end{align*}
Assume now $z\in (2h,1/4)$ and apply Lemma~\ref{lem:4.1} to $w$ in the domains $D_{2z}^{bot}$ and $D_{2z}^{top}$ get for some $a_1,a_2\in\mathbb R,$
\begin{equation}
\label{4.19}
\int_{D_{2z}^{bot}} |\partial_y w-a_1|^p\leq \frac{Cz^p}{h^p}\int_{D_{2z}^{bot}} |\partial_x w|^p,\qquad
\int_{D_{2z}^{top}} |\partial_y w-a_2|^2\leq \frac{Cz^p}{h^p}\int_{D_{2z}^{top}} |\partial_x w|^p.
\end{equation}
Next the estimate (\ref{4.20}) and the inequality $(s+t)^p\leq 2^{p-1}(s^p+t^p)$ for $s,t\geq 0$, that
\begin{equation}
\label{4.20}
 hz(a_1^p+a_2^p)\leq C\int_{I_h\times\left((z,2z)\cup(1-2z,1-z)\right)}|\partial_y w|^p+\frac{Cz^p}{h^p}\int_{D_{2z}^{bot}} |\partial_x w|^p+\frac{Cz^p}{h^p}\int_{D_{2z}^{top}}|\partial_x w|^p.
\end{equation}
Next we rewrite (\ref{4.18}) for $z/2\in (h,1/8)$ instead of $z$ to have
\begin{align}
\label{4.21}
\left(\int_{I_h\times(z,1-z)} |\partial_y w|^p \right)^2& \leq C\left(\frac{1}{h}\|w\|_p\|\partial_x w\|_p\right)^p+\frac{C}{z^{2p}}\left(1+\frac{1}{\epsilon^2}\right)\|w\|_p^{2p}\\ \nonumber
&+C\epsilon^2\left(\int_{I_h\times(z/2,z)} |\partial_y w|^p \right)^2+C\epsilon^2\left(\int_{I_h\times(1-z,1-z/2)} |\partial_y w|^p \right)^2.
\end{align}
Now the obvious bound 
$$\int_{I_h\times\left((z,2z)\cup(1-2z,1-z)\right)}|\partial_y w|^p\leq \int_{I_h\times(z,1-z)} |\partial_y w|^p $$
implies together with (\ref{4.20}) and (\ref{4.21}) the estimate 
\begin{align}
\label{4.22}
(hz(a_1^p+a_2^p))^2 &\leq  C\left(\frac{1}{h}\|w\|_p\|\partial_x w\|_p\right)^p+\frac{Cz^{2p}}{h^{2p}}\int_D |\partial_x w|^p+\frac{C}{z^{2p}}\left(1+\frac{1}{\epsilon^2}\right)\|w\|_p^{2p}\\ \nonumber
&+C\epsilon^2\left(\int_{I_h\times(z/2,z)} |\partial_y w|^p \right)^2+C\epsilon^2\left(\int_{I_h\times(1-z,1-z/2)} |\partial_y w|^p \right)^2.
\end{align}
Thus recalling the triangle inequality the estimate (\ref{4.22}) will give in combination with (\ref{4.19}) the bound 
\begin{equation}
\label{4.23}
(hz(a_1^p+a_2^p))^2[1-C\epsilon^2] \leq  C\left(\frac{1}{h}\|w\|_p\|\partial_x w\|_p\right)^p+\frac{Cz^{2p}}{h^{2p}}\|\partial_x w\|_p^{2p}+\frac{C}{z^{2p}}\left(1+\frac{1}{\epsilon^2}\right)\|w\|_p^{2p},
\end{equation}
for some $C>0.$ Choosing now  $\epsilon=\frac{1}{2\sqrt{C}}$ in (\ref{4.23}) we obtain
\begin{equation}
\label{4.24}
(hz(a_1^p+a_2^p))^2 \leq  C\left(\frac{1}{h}\|w\|_p\|\partial_x w\|_p\right)^p+\frac{Cz^{2p}}{h^{2p}}\|\partial_x w\|_p^{2p}+\frac{C}{z^{2p}}\|w\|_p^{2p}.
\end{equation}
Next we combine (\ref{4.24}), (\ref{4.19}), and the triangle inequality to get 
\begin{equation}
\label{4.25}
\left(\int_{D_{2z}^{bot}\cup D_{2z}^{top}}|\partial_y w|^p\right)^2\leq C\left(\frac{1}{h}\|w\|_p\|\partial_x w\|_p\right)^p+\frac{Cz^{2p}}{h^{2p}}\|\partial_x w\|_p^{2p}+\frac{C}{z^{2p}}\|w\|_p^{2p}.
\end{equation}
Finally putting together (\ref{4.25}) and (\ref{4.18}) we arrive at the key interior estimate
\begin{equation}
\label{4.26}
\left(\int_{I_h\times (0,1)}|\partial_y w|^p\right)^2\leq C\left(\frac{1}{h}\|w\|_p\|\partial_x w\|_p\right)^p+\frac{Cz^{2p}}{h^{2p}}\|\partial_x w\|_p^{2p}+\frac{C}{z^{2p}}\|w\|_p^{2p}.
\end{equation}
In order to derive (\ref{4.9}) from (\ref{4.26}) one needs to minimize the right hand side or (\ref{4.26}) in $z>0$ subject to the constraints $2h<z<1/4.$ This is straightforward and is done by 
considering the relative placement of the global minimum point $z_0=\left(\frac{h\|w\|}{\|\partial_x w\|}\right)^{1/2}$ to the endpoints $2h$ and $1/4.$ We omit the details here that can be found in 
[\ref{bib:Harutyunyan.3}]. 

\end{proof}

\textbf{Step 2. Estiates on $D_h^{top}$ and $D_h^{bot}$.} \textit{There exists a constant $C>0$ such that any harmonic function $w\in C^2(D)$ fulfills inequality:}
\begin{equation}
\label{4.27}
\|\partial_y  w\|_{L^p\left(D_h^{bot}\cup D_h^{top}\right)}^2\leq C\left(\frac{1}{h}\|w\|_p\cdot\|\partial_x w\|_p+\|w\|_p^2+\||\partial_x  w\|_p^2\right).
\end{equation}
\begin{proof}[Proof of Step 2.] 

The proof here follows from minimization of the right-hand-side of (\ref{4.25}) in $z\in (2h,1/4)$ as described above for (\ref{4.26}), as well as utilization of the inequality $z>h$ for the 
left-hand-side of (\ref{4.25}). 
\end{proof}

Denote next for $z\in (0,1/2)$ the middle part 
$$
D_{z}^{mid}=\{(x,y) \ : \ y\in(z,1-z), x\in (-\varphi_1(y),\varphi_2(y))\}.
$$

\textbf{Step 3. Estimate on $D_h^{mid}.$} \textit{There exists a constant $C>0$ such that any harmonic function $w\in C^2(D)$ fulfills inequality:}
\begin{equation}
\label{4.28}
\|\partial_y  w\|_{L^p(D_h^{mid})}^2\leq C\left(\frac{1}{h}\|w\|_p\cdot\|\partial_x w\|_p+\|w\|_p^2+\||\partial_x  w\|_p^2\right).
\end{equation}
\begin{proof}[Proof of Step 3.]

The proof of this step too follows the lines of [\ref{bib:Harutyunyan.3}], with the $L^p$ versions of Lemmas 4.3 and 4.4 in [\ref{bib:Harutyunyan.3}], that we formulate for the convenience of the reader. Recall Lemma~3 from [\ref{bib:Kon.Ole.1}]:
\begin{lemma}
\label{lem:4.3}
Assume $\Omega\subset \mathbb R^n$ is a bounded domain. Then there exists a constant $C_p>0$ depending only on $p,$ such that the estimate 
$$\int_{\Omega}\rho^p |\nabla v|^p dx\leq C_p\int_\Omega |v|^pdx$$ 
holds for any harmonic in $\Omega$ function $v,$ where $\rho(x)$ is the distance function from the boundary of $\Omega.$  
\end{lemma}

\begin{lemma}
\label{lem:4.4}
Assume $\lambda\in (0,1)$, $a<b$ and $f\colon[a,b]\to\mathbb R$ is absolutely continuous. Then the inequality holds:
$$
\int_{a+\lambda(b-a)}^b |f(t)|^pdt\leq \frac{2+\lambda}{\lambda}\int_{a}^{a+\lambda(b-a)}|f(t)|^pdt+2^p(p-1)^{p-1}\int_a^{b}(b-t)^p|f'(t)|^pdt.
$$
\end{lemma}

\begin{proof}
By change of variables we can assume without loss of generality that $a=0.$ For any $x\in [\lambda b,b]$ we have integrating by parts and by Young's ineaqulity 
\begin{align}
\label{4.29}
\int_{0}^x |f(t)|^pdt&\leq x|f(x)|^p+p\int_{0}^x t|f(t)|^{p-1}|f'(t)|dt\\ \nonumber
&\leq x|f(x)|^p+(p-1)\epsilon^{p/(p-1)}\int_{0}^x |f(t)|^{p}dt+\frac{1}{\epsilon^p} \int_{0}^x |tf'(t)|^{p}dt\\ \nonumber
&\leq x|f(x)|^p+\frac{1}{2} \int_{0}^b |f(t)|^{p}dt+(2(p-1))^{p-1}\int_{0}^x |tf'(t)|^{p}dt,
\end{align}
where we have chosen $\epsilon=(2(p-1))^{(1-p)/p}.$  Next, by the mean value theorem, the value of $x$ can be 
chosen so that $|f(x)|^p=\frac{1}{b(1-\lambda)}\int_{\lambda b}^b |f(t)|^pdt,$ thus we have 
\begin{equation}
\label{4.30}
x|f(x)|^p\leq \frac{1}{1-\lambda}\int_{\lambda b}^b |f(t)|^pdt.
\end{equation}
Putting together (\ref{4.29}) and (\ref{4.30}) and keeping in mind that $\lambda b\leq x\leq b,$ we get the estimate 

\begin{equation}
\label{4.31}
\int_{0}^{\lambda b} |f(t)|^pdt \leq  \frac{3-\lambda }{1-\lambda}\int_{\lambda b}^b |f(t)|^pdt+2^p(p-1)^{p-1}\int_{0}^b |tf'(t)|^{p}dt
\end{equation}
It remains to note that the change of variables $t=b-x$ and $\mu=1-\lambda$ in (\ref{4.31}) completes the proof of the lemma. 

\end{proof}

Note that we have from (\ref{4.9}) the obvious estimate 
\begin{equation}
\label{4.32}
\|\partial_y w\|_{L^p\left((-h/2,h/2)\times(h,1-h)\right)}^2\leq C\left(\frac{1}{h}\|w\|_p\cdot\|\partial_x w\|_p+\|w\|_p^2+\|\partial_x w\|_p^2\right),
\end{equation}
which we aim to extend in the horizontal direction to the entire set $D_h^{mid}$ by means of Lemmas 4.3 and 4.4. To that end we fix a point $y\in (h,1-h)$ and apply Lemma~\ref{lem:4.4} to the function $\partial_y w(x,y)$ on the segment $[0,\varphi_2(y)]$ for the value $\lambda=\frac{h}{2\varphi_2(y)}.$ Thus we have  
$$
\int_{\frac{h}{2}}^{\varphi_2(y)}|\partial_y  w(x,y)|^pdx\leq C\left(\int_{0}^{\frac{h}{2}}|\partial_y  w(x,y)|^pdx+C\int_{0}^{\varphi_2(y)}|(\varphi_2(y)-x)\partial^2_{xy} w(x,y)|^pdx \right)
$$
thus integrating in $y$ over $(h,1-h)$ we get
\begin{equation}
\label{4.33}
\int_{T}|\partial_y w|^p\leq  C\left(\int_{(0,h/2)\times(h,1-h)}|\partial_y w|^p+\int_{T}|(\varphi_2(y)-x)\partial^2_{xy}w|^p\right),
\end{equation}
where $T=\{(x,y) \ : \ y\in(h,1-h), x\in(0,\varphi_2(y))\}.$ Observe that $\partial_x w$ is harmonic in $D$ too. On the other hand due to the bounds (\ref{4.1}) and the definition of $T,$ 
we have $|\varphi_2(y)-x|\leq C\delta(x,y),$ where $\delta(x,y)$ is the distance function from the boundary of $D.$ Hence Lemma~\ref{lem:4.3} gives the bound
$$
\int_{T}|(\varphi_2(y)-x)\partial^2_{xy}w|^p\leq C \int_{D}| \delta \nabla (\partial_xw)|^p \leq C \int_{D}|\partial_xw|^p,
$$
which yields invoking (\ref{4.33}) the bound
\begin{equation}
\label{4.34}
\int_{T}|\partial_y w|^2\leq C\left(\int_{(0,h/2)\times(h,1-h)}|\partial_y w|^p+C\int_{D}|\partial_x w|^p\right).
\end{equation}
An analogous estimate for the left half of $D_{mid}^h$ is straightforward. Thus combining the two with (\ref{4.9}) we get (\ref{4.28}).
\end{proof}
In the final step we put together (\ref{4.27}) and (\ref{4.28}) to get (\ref{4.8}). 

\end{proof}

\section{Proof of the main results}
\label{sec:5}
\setcounter{equation}{0}

\begin{proof}[Proof of Theorem~\ref{th:3.1}]
As already mentioned we will prove the estimate (\ref{3.2}) first for $\BF$ in place of $\nabla \Bu$. The proof will follow the appropriate lines of [\ref{bib:Harutyunyan.4}], skipping the exactly identical calculations and presenting the modifications when necessary. Also, in what follows the norm $\|\cdot\|$ will be the $L^p$ norm $\|\cdot\|_{L^p(\Omega)}$ unless specified otherwise. We prove the estimate (\ref{3.2}) block by block by freezing each of the variables $t,$ $\Gth$ and $z$ and considering the appropriate inequality on the $t,\Gth,z=$const cross sections of $\Omega.$ The lemma below is a key estimate for the blocks $z=const$ and $\Gth=const.$ 
\begin{lemma}
\label{lem:5.1}
Let $h,b>0$ with $h\in (0,b/8)$ and assume the Lipschitz functions $\varphi_1,\varphi_2\colon[0,b]\to(0,\infty)$ satisfy the usual uniform conditions 
\begin{equation}
\label{5.1}
h\leq \varphi_i(y)\leq C_1h,\quad |\nabla\varphi_i(y)|\leq C_2h,\quad\text{for all}\quad y\in[0,b], i=1,2.
\end{equation}
Denote the two dimensional thin domain $D=\{(x,y) \ : \  y\in(0,b), x\in (-\varphi_1(y), \varphi_2(y))\}.$ Given a displacement field
$\BU=(u(x,y),v(x,y))\in W^{1,p}(D,\mathbb R^2),$ the vector fields $\BGa,\BGb\in W^{1,\infty}(D,\mathbb R^2)$ and a function $w\in W^{1,p}(D,\mathbb R),$ denote
\begin{equation}
\label{5.2}
\BM=
\begin{bmatrix}
u_{,x} & u_{,y}+\BGa\cdot\BU\\
v_{,x} & v_{,y}+\BGb\cdot\BU+w
\end{bmatrix}.
\end{equation}
Then for any $\epsilon \in (0,1)$ the following Korn-like interpolation inequality holds:
\begin{equation}
\label{5.3}
\|\BM\|_p^2\leq C\left(\frac{\|u\|_p\cdot \|\BM^{sym}\|_p}{h}+\|\BM^{sym}\|_p^2+\left(\frac{1}{\epsilon}+h^2\right)\|\BU\|_p^2+(\epsilon+h^2)(\|w\|_p^2+\|\partial_x w\|_p^2)\right),
\end{equation}
for all $h\in(0,h_0).$ Here $\BM^{sym}=\frac{1}{2}(\BM+\BM^T)$ is the symmetric part of $\BM,$ the constants $C$ and $h_0$ depend only on the quantities $p$, $b,$ $\|\BGa\|_{W^{1,\infty}},$ $\|\BGb\|_{W^{1,\infty}},$ and the norm $\|\cdot\|_p$ is the $L^p$ norm $\|\cdot\|_{L^p(D)}.$
\end{lemma}

\begin{proof}
Let us point out that the proof for the case $p=2$ does not go through in its exact form in [\ref{bib:Harutyunyan.3}] and one needs to make some modifications, thus we present a complete proof here. First of all, we can assume by density that $\BU\in C^2(\bar D).$ For functions $f,g\in W^{1,p}(D,\mathbb R)$ denote by $\BM_{f,g}$ the matrix
\begin{equation}
\label{5.4}
\BM_{f,g}=
\begin{bmatrix}
u_{,x} & u_{,y}+f\\
v_{,x} & v_{,y}+g
\end{bmatrix}.
\end{equation}
Assume $\tilde u(x,y)$ is the harmonic part of $u$ in $D,$ i.e., it is the unique solution of the Dirichlet boundary value problem
\begin{equation}
\label{5.5}
\begin{cases}
\Delta \tilde u(x,y)=0, & (x,y)\in D\\
\tilde u(x,y)=u(x,y), & (x,y)\in \partial D.
\end{cases}
\end{equation}
As $u-\tilde u$ vanishes on the lateral boundary of $D,$ then we have by the Poincar\'e inequality in the horizontal direction, that
\begin{equation}
\label{5.6}
\|u-\tilde u\|_p\leq Ch\|\nabla(u-\tilde u)\|_p.
\end{equation}
A simple calculation gives the identity 
\begin{equation}
\label{5.7}
\Delta (u-\tilde u)=\Delta u=\frac{\partial }{\partial x} ((\BM_{f,g}^{sym})_{11}-(\BM_{f,g}^{sym})_{22})+2\frac{\partial }{\partial y}(\BM_{f,g}^{sym})_{12}+
\frac{\partial g}{\partial x}-\frac{\partial f}{\partial y}.
\end{equation}
Assume next the function $w\in W^{1,p}(D,\mathbb R)$ is the unique solution of the Dirichlet boundary value problem
\begin{equation}
\label{5.8}
\begin{cases}
\Delta w(x,y)=\frac{\partial g}{\partial x}-\frac{\partial f}{\partial y}, & (x,y)\in D\\
w(x,y)=0, & (x,y)\in \partial D.
\end{cases}
\end{equation}
Introducing the vector field 
$$\BV=\left ((\BM_{f,g}^{sym})_{11}-(\BM_{f,g}^{sym})_{22}+\frac{\partial w}{\partial x} ,2(\BM_{f,g}^{sym})_{12}+\frac{\partial w}{\partial y} \right),$$
we get from (\ref{5.7}) and (\ref{5.8}) that $u-\tilde u$ solves the problem
\begin{equation}
\label{5.9}
\begin{cases}
\Delta (u-\tilde u)=\mathrm{div}\BV, & (x,y)\in D\\
u-\tilde u=0, & (x,y)\in \partial D,
\end{cases}
\end{equation}
thus we get by the classical $L^p$ estimates for linear elliptic $PDEs$ [\ref{bib:Meyers}],
\begin{equation}
\label{5.10}
\|\nabla(u-\tilde u)\|_p\leq C\|\BV\|_p\leq C(\|\BM_{f,g}^{sym}\|_p+\|\nabla w\|_p).
\end{equation}
Next we have by the Poincar\'e inequality in the $x$ direction, that
\begin{equation}
\label{5.11}
\|w\|_p \leq Ch\|\nabla w\|_p,
\end{equation}
thus an application of the interpolation inequality $\|\nabla w\|_p^2 \leq C\|w\|_p\|\Delta w\|_p$
together with (\ref{5.8}) derives from (\ref{5.11}) the estimate
\begin{equation}
\label{5.12}
\|\nabla w\|_p \leq Ch\|\Delta w\|_p\leq Ch(\|\partial_x g\|_p+\|\partial_y f\|_p).
\end{equation}
Combining (\ref{5.6}), (\ref{5.10}) and (\ref{5.12}) we obtain the key estimates
\begin{align}
\label{5.13}
\|\nabla(u-\tilde u)\|_p&\leq C\left[\|\BM_{f,g}^{sym}\|_p+h(\|\partial_x g\|_p+\|\partial_y f\|_p)\right],\\ \nonumber
\|u-\tilde u\|_p&\leq Ch\left[\|\BM_{f,g}^{sym}\|_p+h(\|\partial_x g\|_p+\|\partial_y f\|_p)\right].
\end{align}
By the harmonicity of $\tilde u$ we can apply Lemma~\ref{lem:4.2} to it, hence doing so we have due to the bounds (\ref{5.13}),
\begin{align}
\label{5.14}
\|\partial_y u+f\|_p^2&\leq 4(\|\partial_y u-\partial_y\tilde u\|_p^2+\|\partial_y \tilde u\|_p^2+\|f\|_p^2)\\ \nonumber
&\leq C\left(\|\nabla(u-\tilde u)\|_p^2+\frac{1}{h}\|\tilde u\|_p\cdot\|\partial_x \tilde u\|_p+\|\tilde u\|_p^2+\|\partial_x \tilde u\|_p^2+\|f\|_p^2\right)\\ \nonumber
&\leq C\left(\|\nabla(u-\tilde u)\|_p^2+\frac{1}{h}(\|u\|_p+\|u-\tilde u\|_p)(\|\BM_{f,g}^{sym}\|_p+\|\nabla(u-\tilde u)\|_p)\right)\\ \nonumber
&+C\left(\|u\|_p^2+\|u-\tilde u\|_p^2+\|\BM_{f,g}^{sym}\|_p^2+\|\nabla(u-\tilde u)\|_p^2+\|f\|_p^2\right)\\ \nonumber
&\leq C\left(\frac{1}{h}\|u\|_p\cdot\|\BM_{f,g}^{sym}\|_p+\left(1+\frac{1}{\epsilon}\right)\|u\|_p^2+(\epsilon+h^2)(\|\partial_y f\|_p^2+\|\partial_x g\|_p^2)\right)\\ \nonumber
&+C(\|\BM_{f,g}^{sym}\|_p^2+\|f\|_p^2).
\end{align}
Recall that in our case we have $f=\BGa\cdot\BU$ and $g=\BGb\cdot\BU+w,$ thus the obvious bounds hold:
\begin{align}
\label{5.15}
\|\partial_y f\|_p+\|\partial_x g\|_p&\leq C\|\BU\|_{W^{1,p}(D)}+\|\partial_x w\|_p\\ \nonumber
&\leq C(\|\BM_{f,g}^{sym}\|_p+\|\BU\|_p+\|w\|_p+\|\partial_x w\|_p),\\ \nonumber
\|f\|_p& \leq C\|\BU\|_p.
\end{align}
Consequently, (\ref{5.3}) follows from (\ref{5.14}) and (\ref{5.15}) by several applications of the triangle inequality. 
\end{proof}

\textbf{The block $\Gth=const$.} We aim to prove that for any $\epsilon>0$ and for small enough $h$ the estimate holds:
\begin{equation}
\label{5.16}
\|F_{13}\|_p^2+\|F_{31}\|_p^2\leq C\left(\frac{\|u_t\|_p\cdot\|\BF^{sym}\|_p}{h}+\|\BF^{sym}\|_p^2+\frac{1}{\epsilon}\|\Bu\|_p^2+\epsilon \|F_{12}\|_p^2\right),
\end{equation}
where $\|\cdot\|_p=\|\cdot\|_{L^p(\Omega)}.$ 

\begin{proof}
The proof is achieved by an application of Lemma~\ref{lem:5.1} to a suitably chosen set of vector fields $\BU,\BGa,\BGb$ and a vector $w.$ 
Recall that the simplified gradient matrix $\BF$ is given by omitting the small terms in the denominators of the second 
and third columns of $\nabla \Bu$ in (\ref{2.2}), i.e., 
\begin{equation}
\label{5.17}
\BF=
\begin{bmatrix}
  u_{t,t} & \dfrac{u_{t,\Gth}-A_{\Gth}\Gk_{\Gth}u_{\Gth}}{A_{\Gth}} &
\dfrac{u_{t,z}-A_{z}\Gk_{z}u_{z}}{A_{z}}\\[3ex]
u_{\Gth,t}  &
\dfrac{A_{z}u_{\Gth,\Gth}+A_{z}A_{\Gth}\Gk_{\Gth}u_{t}+A_{\Gth,z}u_{z}}{A_{z}A_{\Gth}} &
\dfrac{A_{\Gth}u_{\Gth,z}-A_{z,\Gth}u_{z}}{A_{z}A_{\Gth}}\\[3ex]
u_{ z,t}  & \dfrac{A_{z}u_{z,\Gth}-A_{\Gth,z}u_{\Gth}}{A_{z}A_{\Gth}} &
\dfrac{A_{\Gth}u_{z,z}+A_{z}A_{\Gth}\Gk_{z}u_{t}+A_{z,\Gth}u_{\Gth}}{A_{z}A_{\Gth}}
\end{bmatrix},
\end{equation}
Hence, an applicable choice turns out to be 
$\BU=(u_t,A_zu_z)$ with the vector fields $\BGa=(0,-\Gk_z),$ $\BGb=(A_z^2\Gk_z,-\frac{A_{z,z}}{A_z})$ and the function $w=\frac{A_zA_{z,\Gth}}{A_\Gth}u_\Gth$ in the variables $t$ and $z.$ Indeed, we can calculate owing back to formula (\ref{5.2}) that 
\begin{align}
\label{5.18}
M_{11}&=u_{t,t}\\ \nonumber
M_{12}&=u_{t,z}+(0,-\Gk_z) \cdot (u_t,A_zu_z)\\ \nonumber
&=u_{t,z}-A_z\Gk_zu_z\\ \nonumber
M_{21}&=A_zu_{z,t}\\ \nonumber
M_{22}&=(A_zu_z)_{,z}+(A_z^2\Gk_z,-\frac{A_{z,z}}{A_z})\cdot (u_t,A_zu_z)+\frac{A_zA_{z,\Gth}}{A_\Gth}u_\Gth\\ \nonumber
&=A_zu_{z,z}+A_z^2\Gk_zu_t+\frac{A_zA_{z,\Gth}}{A_\Gth}u_\Gth,\\ \nonumber
\partial_t w&=\frac{A_zA_{z,\Gth}}{A_\Gth}u_{\Gth,t}\\ \nonumber
\end{align}
thus taking into account formula (\ref{5.17}) we have 
\begin{align}
\label{5.19}
\BM^{sym}_{11}&=\BF^{sym}_{11},\quad \BM^{sym}_{12}=A_z\BF^{sym}_{12},\quad \BM^{sym}_{22}=A_z^2\BF^{sym}_{33},\\ \nonumber
|w|&\leq C|\Bu|,\quad |\partial_t w|\leq C(|F_{12}|+|\Bu|).
\end{align}
The rest follows from integrating the obtained estimates (from Lemma~\ref{lem:5.1}) in $\Gth\in(0,\omega)$ and applying H\"oldre inequality to the product terms.  
\end{proof}

\textbf{The block $z=const$.} The role of the variables $\Gth$ and $z$ is the completely the same, thus we have an analogous estimate
\begin{equation}
\label{5.20}
\|F_{12}\|_p^2+\|F_{21}\|_p^2\leq C\left(\frac{\|u_t\|_p\cdot\|\BF^{sym}\|_p}{h}+\|\BF^{sym}\|_p^2+\frac{1}{\epsilon}\|\Bu\|_p^2+\epsilon \|F_{31}\|_p^2\right).
\end{equation}
Consequently adding (\ref{5.16}) and (\ref{5.20}) and choosing the parameter $\epsilon>0$ small enough (independent of $h$) we discover
\begin{equation}
\label{5.21}
\|F_{12}\|_p^2+\|F_{21}\|_p^2+\|F_{13}\|_p^2+\|F_{31}\|_p^2\leq C\left(\frac{\|u_t\|_p\cdot\|\BF^{sym}\|_p}{h}+\|\BF^{sym}\|_p^2+\|\Bu\|_p^2\right).
\end{equation}
\textbf{The block $t=const$.} As in [\ref{bib:Harutyunyan.3}], we will prove an estimate on the shell
\begin{equation}
\label{5.22}
\Omega^h=\{(t,\Gth,z)\in\Omega \ : \ t\in (-h,h)\}
\end{equation}
and then extend it to $\Omega$ in the normal direction by means of a localization argument. Namely, we prove that
\begin{equation}
\label{5.23}
\|F_{23}\|_{L^p(\Omega^h)}+\|F_{32}\|_{L^p(\Omega^h)}\leq C(\|\Bu\|_{L^p(\Omega^h)}+\|\BF^{sym}\|_{L^p(\Omega^h)}).
\end{equation}

\begin{proof}

\begin{lemma}
\label{lem:5.2}
Let $E=\{(\Gth,z)\ : \ \Gth\in (0,\omega), \ z\in (z^1(\Gth),z^2(\Gth))\}$  and assume $\varphi=\varphi(\Gth,z)\in C^1(E,\mathbb R)$ satisfies the conditions
\begin{equation}
\label{5.24}
0<c_1\leq \varphi(\Gth,z)\leq c_2,\quad |\nabla \varphi(\Gth,z)|\leq c_3,\quad\text{for all}\quad (\Gth,z)\in E.
\end{equation}
For a displacement $\BU=(u,v)\in W^{1,p}(E,\mathbb R^2)$ set
\begin{equation}
\label{5.25} \BM_\varphi=
\begin{bmatrix}
u_{,\Gth} & \varphi u_{,z}\\
v_{,\Gth} & \varphi v_{,z}
\end{bmatrix}.
\end{equation}
Then there exists a constant $c>0,$ depending only on parameters of the domain mid-surface $S$ and $c_i,\ i=1,2,3,$ such that
\begin{equation}
\label{5.26}
\|\BM_\varphi\|_{L^p(E)}\leq c(\|\BM_\varphi^{sym}\|_{L^p(E)}+\|\BU\|_{L^p(E)}).
\end{equation}
\end{lemma}

\begin{proof}
The proof follows from Korn's second inequality [\ref{bib:Kon.Ole.2}, Theorem~2], which reads as 
$\|\nabla \BW\|_{L^p(E)}\leq C(\|e(\BW)\|_{L^p(E)}+\|\BW\|_{L^p(E)}),$ applied to the auxiliary 
field $\BW=(u,\frac{1}{\varphi}v)\colon E\to\mathbb R^2.$ The details are omitted here.
\end{proof}

Estimate (\ref{5.23}) now follows from applying Lemma~\ref{lem:5.2} to the displacement field $\BU=(u_\Gth,u_z),$ with the choice 
$\varphi(\Gth,z)=\frac{A_\Gth}{A_z}.$ We refer the reader to [\ref{bib:Harutyunyan.3}] for details. 

\end{proof}

Next we combine the estimates (\ref{5.21}) and (\ref{5.23}) to get the bound 
\begin{equation}
\label{5.27}
\|\BF\|_{L^p(\Omega^h)}^2 \leq C\left(\frac{\|u_t\|_{L^p(\Omega)}\cdot\|\BF^{sym}\|_{L^p(\Omega)}}{h}+\|\BF^{sym}\|_{L^p(\Omega)}^2+\|\Bu\|_{L^p(\Omega)}^2\right).
\end{equation}
As already pointed out, it is easy to see, that by an application of the obvious bounds
$\|\BF-\nabla\Bu\|\leq h\|\nabla\Bu\|$ and $\|\BF^{sym}-e(\Bu)\|\leq h\|\nabla\Bu\|,$ we obtain from (\ref{5.27}) the partial estimate 
\begin{equation}
\label{5.28}
\|\nabla\Bu\|_{L^p(\Omega^h)}^2 \leq C\left(\frac{\|u_t\|_{L^p(\Omega)}\cdot\|e(\Bu)\|_{L^p(\Omega)}}{h}
+\|e(\Bu)\|_{L^p(\Omega)}^2+\|\Bu\|_{L^p(\Omega)}^2\right),
\end{equation}
for small enough $h.$ The $L^2$ version of the following lemma has been proven in [\ref{bib:Harutyunyan.3}, Lemma~5.2], the $L^p$ analog is completely analogous, the point is that Korn's first inequality holds in $L^p$ too, thus we will skip the proof. 
\begin{lemma}
\label{lem:5.3}
Assume $D_1\subset D_2\subset \mathbb R^n$ are open bounded connected Lipschitz domains. By Korn's first inequality, there exist constants $K_1$ and $K_2$ such that for any vector field $\BU\in W^{1,p}(D_2,\mathbb R^n),$ there exist skew-symmetric matrices $\BA_1,\BA_2\in\mathbb M^{n\times n},$ such that 
\begin{equation}
\label{5.29}
\|\nabla\BU-\BA_1\|_{L^p(D_1)} \leq K_1\|e(\BU)\|_{L^p(D_1)},\quad \|\nabla\BU-\BA_2\|_{L^p(D_2)} \leq K_2\|e(\BU)\|_{L^p(D_2)}.
\end{equation}
The assertion is that there exists a constant $C>0$ depending only on the quantities $K_1,K_2$ and $\frac{|D_2|}{|D_1|},$ such that for any 
vector field $\BU\in W^{1,p}(D_2, \mathbb R^n)$ one has
\begin{equation}
\label{5.30}
\|\nabla\BU\|_{L^p(D_2)} \leq C(\|\nabla\BU\|_{L^p(D_1)}+\|e(\BU)\|_{L^p(D_2)}).
\end{equation}
\end{lemma}
The idea is now to divide $\Omega$ into small parts with size of order $h$ and extend the existing local estimate on all smaller parts in the normal to $S$ direction to the bigger (but still of order $h$) parts containing it. Assume now $\bar\Bu=(\bar u_1,\bar u_2, \bar u_3)$ is $\Bu$ in Cartesian coordinates $\Bx=(x_1,x_2,x_3)$ and denote by $\bar\nabla$ the cartesian gradient. we divide the domains $\Omega$ and $\Omega^h$ into small pieces of order $h.$ Namely for $N=[\frac{1}{h}]+1$ denote 
\begin{align}
\label{5.31}
&\Omega_{i,j}=\left\{(t,\Gth,z)\in\Omega \ : \ \Gth\in \left(\frac{i}{N},\frac{i+1}{N}\right), z\in \left(\frac{Nz_1+j(z_2-z_1)}{N}, \frac{Nz_1+(j+1)(z_2-z_1)}{N}\right)\right\},\\ \nonumber
&\Omega_{i,j}^h=\{(t,\Gth,z)\in\Omega_{i,j} \ : \ t\in (-h,h)\} ,\quad   i,j=0,1,\dots,N-1.
\end{align}
Note that as Korn's first inequality is invariant under the variable change $x\to\lambda x,$ then so is the constant $C$ in (\ref{5.30}). Second, the domains $D_1=\Omega_{i,j}$ and $D_2=\Omega_{i,j}^h$ have uniform Lipschitz constants depending only on the parameters mid-surface $S$ and the functions $g_1^h, g_2^h,$ are of order $h,$ thus by the above remark and Lemma~\ref{lem:5.2} we have the estimate
\begin{equation}
\label{5.32}
\|\bar\nabla\bar\Bu\|_{L^p(\Omega_{ij})} \leq C(\|\bar\nabla\bar\Bu\|_{L^p(\Omega_{ij}^h)}+\|e(\bar\Bu)\|_{L^p(\Omega_{ij})}),\quad   i,j=0,1,\dots,N-1,
\end{equation} 
summing which over $i,j=0,1,\dots,N-1$ we arrive at 
\begin{equation}
\label{5.33}
\|\bar\nabla\bar\Bu\|_{L^p(\Omega)} \leq C(\|\bar\nabla\bar\Bu\|_{L^p(\Omega^h)}+\|e(\bar\Bu)\|_{L^p(\Omega)}).
\end{equation} 
It remains to notice that (\ref{5.28}) and the analogous estimate with $\Bu$ and $\nabla$ replaced by $\bar\Bu$ and $\bar\nabla$ respectively are equivalent, thus (\ref{5.28}) and (\ref{5.33}) yield (\ref{3.2}). The Ansatz proving the sharpness of (\ref{3.2}) and (\ref{3.3}) has been constructed in [\ref{bib:Harutyunyan.1}] and reads as:
\begin{equation}
\label{6.1}
\begin{cases}
u_t=W(\frac{\Gth}{\sqrt{h}},z)\\
u_\Gth=-\frac{t\cdot W_{,\Gth}\left(\frac{\Gth}{\sqrt h},z\right)}{A_\Gth{\sqrt h}}\\
u_z=-\frac{t\cdot W_{,z}\left(\frac{\Gth}{\sqrt h},z\right)}{A_z},
\end{cases}
\end{equation}
where $W(\xi,\eta)\colon\mathbb R^2\to\mathbb R$ is a fixed smooth compactly supported function. The calculations that verify that the displacement field 
$\Bu=(u_y, u_\Gth, u_z)$ realizes the asymptotics of $h$ in both (\ref{3.2}) and (\ref{3.3}) are left to the reader. 
\end{proof}

\section*{Acknowledgements}
This material is based upon work supported by the National Science Foundation under Grants No. DMS-1814361.

\end{document}